\documentclass[12pt]{amsart}
\usepackage{amssymb}
\usepackage{amsmath}
\usepackage{longtable}
\newcommand\g{{\mathfrak g}}

\newcommand\m{\mathfrak m}
\newcommand\z{\mathfrak z}

\newcommand\q{\mathfrak q}

\newcommand\W{{\bf A}}
\newcommand\K{\mathbb K}
\newcommand\U{\mathcal U}

\newcommand\Id{\mathfrak{Id}}

\newcommand\Walg{\mathcal W}

\newcommand\A{\mathcal A}

\newcommand\gr{\operatorname{gr}}

\newcommand\Orb{\mathbb{O}}
\newcommand\I{\mathcal I}
\newcommand\J{\mathcal J}
\renewcommand\sl{\mathfrak{sl}}
\newcommand{\ad}{\mathop{\rm ad}\nolimits}

\newcommand\Centr{\mathcal Z}

\newcommand{\VA}{\operatorname{V}}

\newcommand\Prim{\operatorname{Pr}}
\newcommand\Leaf{\mathcal{L}}
\newtheorem{Thm}{Theorem}[section]
\newtheorem{Prop}[Thm]{Proposition}

\newtheorem{Lem}[Thm]{Lemma}
\theoremstyle{definition}

\newtheorem{Rem}[Thm]{Remark}

\unitlength=1mm   \numberwithin{equation}{section}
\numberwithin{table}{section} \oddsidemargin=0cm
\author{Ivan Losev}
\title{Primitive ideals for W-algebras in type A}
\thanks{Supported by the NSF grant DMS-0900907}
\thanks{Address: Northeastern University, Department of Mathematics, 360 Huntington Avenue, Boston, MA 02115}
\thanks{E-mail: i.loseu@neu.edu}
\begin{document}
\begin{abstract}
In this note we classify the primitive ideals in finite W-algebras
of type $A$.
\end{abstract}
\maketitle
\section{Introduction}
Let $\g$ be a semisimple Lie algebra over an algebraically closed
 field $\K$ of characteristic 0 and $e\in \g$ be a nilpotent element. Then to the pair
$(\g,e)$ one can assign an associative algebra $\Walg$ called the W-algebra.
This algebra was defined in full generality by Premet in \cite{Premet1}.
Two equivalent definitions of W-algebras are provided in Section \ref{SECTION_W}. For other details
on W-algebras the reader is referred to review \cite{ICM}.

One of the reasons to be interested in $\Walg$ are numerous connections between this algebra
and the universal enveloping algebra $\U$ of $\g$. For example, the sets of primitive ideals $\Prim(\Walg)$
and $\Prim(\U)$ of $\Walg$ and $\U$ are closely related. Recall that an ideal in an associative algebra
is called primitive if it is the annihilator of some irreducible module. The structure of
$\Prim(\U)$ was studied extensively in 70's and 80's.

One of manifestations of a relationship between $\Prim(\Walg)$ and $\Prim(\U)$ is a map
$\I\mapsto \I^\dagger:\Prim(\Walg)\rightarrow \Prim(\U)$ constructed in \cite{Wquant}.
One can describe the image of this map. Namely, to each primitive ideal $\J\in \Prim(\U)$
one  assigns its associated variety $\VA(\U/\J)$. According to a theorem of Joseph,  the associated variety is 
the closure of a single nilpotent orbit in $\g^*\cong \g$. Thanks to \cite{Wquant}, Theorem 1.2.2(vii), an element
$\J\in \Prim(\U)$ is of the form $\I^\dagger$ for some $\I\in \Prim(\Walg)$ if and only if $\Orb\subset \VA(\U/\J)$,
where $\Orb$ stands for the adjoint orbit of $e$.

In general, the map $\bullet^\dagger$ is not injective. However, the following result holds.

\begin{Thm}\label{Thm:main}
The map $\bullet^\dagger:\Prim(\Walg)\rightarrow \Prim(\U)$ is an injection provided
$\g\cong \sl_n$.
\end{Thm}

This theorem provides a classification of primitive ideals in $\Walg$ because the set of primitive
ideals  $\J\subset \U$ with  $\VA(\U/\J)=\overline{\Orb}$ is known thanks to the work of Joseph,
\cite{Joseph}.

{\bf Acknowledgements}. I would like to thank Jon Brundan for communicating this problem
to me.

\section{W-algebras and the map between ideals}\label{SECTION_W}
\subsection{Quantum slice}\label{SUBSECTION_quantum_slice}
Let $Y$ be an affine Poisson scheme equipped with a $\K^\times$-action such that
the Poisson bracket has degree $-2$. Let $\A_\hbar$ be an associative flat graded $\K[\hbar]$-algebra
(where $\hbar$ has degree $1$)  such that $[\A_\hbar,\A_\hbar]\subset \hbar^2 \A_\hbar$ and $\K[Y]=\A_\hbar/(\hbar)$ as a graded Poisson
algebra. Pick a point $\chi\in Y$. Let $I_\chi$ be the maximal ideal of $\chi$ in $\K[Y]$ and let
$\widetilde{I}_\chi$ be the inverse image of $I_\chi$ in $\A_{\hbar}$. Consider the completion
$\A_\hbar^{\wedge_\chi}:=\varprojlim_{n\rightarrow \infty} \A_\hbar/\widetilde{I}_\chi^n$.
This is a complete topological $\K[[\hbar]]$-algebra with $\A_\hbar^{\wedge_\chi}/(\hbar)=\K[Y]^{\wedge_\chi}$,
where on the right hand side we have the usual commutative completion. Moreover,
as we have seen in \cite{ES_appendix}, Lemma A2\footnote{In \cite{ES_appendix} we had an assumption that $Y$ has only finitely
many leaves but this assumption is not necessary for many constructions and results.}, the algebra $\A_\hbar^{\wedge_\chi}$ is
flat over $\K[[\hbar]]$.

The cotangent space $T^*_\chi Y=I_\chi/I_\chi^2$ comes equipped with a natural skew-symmetric
form, say $\omega$. Fix a maximal symplectic subspace $V\subset T^*_\chi Y$. One can choose
an embedding $V\hookrightarrow \widetilde{I}_\chi^{\wedge_\chi}$ such that $[\iota(u),\iota(v)]=\hbar^2 \omega(u,v)$
and whose composition with the projection $\widetilde{I}_\chi^{\wedge_\chi}\twoheadrightarrow T_\chi^*Y$
is the identity. This is proved similarly to Proposition 3.3 in \cite{Kaledin} (or can be deduced from that proposition,
compare with the argument of Subsection 7.2 in \cite{Miura}). 
Consider the homogenized Weyl algebra $\W_\hbar(V)=T(V)[\hbar]/(u\otimes v-v\otimes u-\hbar^2 \omega(u,v))$
and its completion $\W_\hbar^{\wedge_0}(V)$ at zero. 
It is easy to show that $$\A^{\wedge_\chi}_\hbar:=\W_\hbar^{\wedge_0}(V)\widehat{\otimes}_{\K[[\hbar]]}\underline{\A}'_\hbar,$$  
where $\underline{\A}'_\hbar$ is the centralizer of $V$  in $\A^{\wedge_\chi}_\hbar$. We remark that the algebra $\underline{\A}'_\hbar$ is complete with respect to the topology induced by its maximal ideal. The symbol ``$\widehat{\otimes}$'' stands for the completed tensor product
of topological vector spaces.

The argument in the proof of \cite{qiso}, Proposition 6.6.1, Step 2,  shows that any two embeddings  $\iota^1,\iota^2:V\hookrightarrow \widetilde{I}_\chi^{\wedge_\chi}$ satisfying the conditions in the previous paragraph differ by an automorphism of $\A_\hbar^{\wedge_\chi}$ of the form $\exp(\frac{1}{\hbar^2}\ad(z))$ with $z\in (\widetilde{I}_\chi^{\wedge_\chi})^3$. In particular, the algebra
$\underline{\A}'_\hbar$ is defined uniquely up to a $\K[[\hbar]]$-linear isomorphism.

Now let us consider a compatibility of our construction with certain derivations.
Suppose $\A_\hbar$ is equipped with a derivation $D$ such that $D\hbar=\hbar$.
The derivation extends to $\A_\hbar^{\wedge_\chi}$. According to \cite{ES_appendix},
there is a derivation $D'$ of $\underline{\A}'_\hbar$ with the following properties.
First, $D'\hbar=\hbar$. Second, we have $D-\tilde{D}'=\frac{1}{\hbar^2}\ad(a)$ for some element
$a\in \A_\hbar^{\wedge_\chi}$, where $\tilde{D}'$ means the derivation that equals $D'$ on $\underline{\A}'_\hbar$
and acts by $1$ on the symplectic space $V$ generating $\W_\hbar^{\wedge_0}(V)$.

An easy special case of the previous construction is when the derivation $D$
comes from a $\K^\times$-action on $\A_\hbar$ preserving $\chi$. Here we can choose a
$\K^\times$-stable $V$ and a  $\K^\times$-equivariant $\iota:V\rightarrow \widetilde{I}_\chi^{\wedge_\chi}$
and so we get a $\K^\times$-action on $\underline{\A}'_\hbar$. Moreover, the
algebra $\underline{\A}'_\hbar$ is now defined uniquely up to a $\K^\times$-equivariant
$\K[[\hbar]]$-linear isomorphism.

Now let $Z(\A_\hbar)$ denote the center of $\A_\hbar$. Consider a $\K^\times$-equivariant
$\K[[\hbar]]$-linear homomorphism $\lambda:Z(\A_\hbar)\rightarrow \K[\hbar]$ and the corresponding central reduction
$\A_{\lambda,\hbar}:=\A_\hbar/ \A_\hbar \ker\lambda$. Until the end of the subsection
we assume that $\chi$ lies in the spectrum of $\A_{\lambda,\hbar}/(\hbar)$.

Consider the induced homomorphism
$Z(\A_\hbar)\rightarrow \A^{\wedge_\chi}_\hbar=\W^{\wedge_0}_\hbar(V)\widehat{\otimes}_{\K[[\hbar]]}\underline{\A}'_\hbar$.
The image is central. The center of $\W^{\wedge_0}_\hbar(V)$ coincides with $\K[[\hbar]]$,
so the image of $Z(\A_\hbar)$ is contained in $\underline{\A}'_\hbar$. Set
$\underline{\A}'_{\lambda,\hbar}=\underline{\A}'_\hbar/\underline{\A}'_\hbar\ker\lambda$.

Then we have the completion $\A^{\wedge_\chi}_{\lambda,\hbar}$ of $\A_{\lambda,\hbar}$, and
$\A^{\wedge_\chi}_{\lambda,\hbar}=\A^{\wedge_\chi}_{\hbar}/\A^{\wedge_\chi}_{\hbar}\ker\lambda$.
Furthermore, we have the following commutative diagram, where the horizontal arrows are isomorphisms
and the vertical arrows are the natural quotients.

\begin{picture}(80,30)
\put(3,22){$\A^{\wedge_\chi}_{\hbar}$}
\put(42,22){$\W^{\wedge_0}_\hbar(V)\widehat{\otimes}_{\K[[\hbar]]}\underline{\A}'_{\hbar}$}
\put(2,2){$\A^{\wedge_\chi}_{\lambda,\hbar}$}
\put(41,2){$\W^{\wedge_0}_\hbar(V)\widehat{\otimes}_{\K[[\hbar]]}\underline{\A}'_{\lambda,\hbar}$}
\put(5,20){\vector(0,-1){13}}
\put(56,20){\vector(0,-1){13}}
\put(12,23.5){\vector(1,0){28}}
\put(12,3.5){\vector(1,0){28}}
\end{picture}

Now suppose that we have a reductive group $Q$ that acts on $\A_\hbar$ rationally by $\K[\hbar]$-algebra automorphism
fixing $\chi$. Further, suppose that there is quantum moment map $\Phi^{\A}: \q\rightarrow \A_\hbar$, i.e.,
a $Q$-equivariant linear map with the property that $[\Phi^{\A}(\xi),a]=\hbar^2 \xi.a$, where on the right hand side
$\xi$ is the derivation of $\A_\hbar$ coming from the $Q$-action. Composing $\Phi^{\A}$ with a natural homomorphism
$\A_\hbar\rightarrow \A^{\wedge_\chi}_\hbar$ we get a quantum moment map $\q\rightarrow \A_\hbar^{\wedge_\chi}$
again denoted by $\Phi^{\A}$.

We remark that we can choose $V$ to be $Q$-stable. This gives rise to a $Q$-action on $V$
by linear symplectomorphisms and hence to an action of $Q$ on $\W_\hbar(V)$ by $\K[\hbar]$-linear
algebra automorphisms. There is a quantum moment map $\Phi^{\W}:\q\rightarrow \W_\hbar(V)$ that
is the composition $\q\rightarrow \mathfrak{sp}(V)=S^2V\hookrightarrow \W_\hbar(V)$.

Further, since $Q$ is reductive, we can assume that the embedding $\iota:V\hookrightarrow \A^{\wedge_\chi}_\hbar$
is $Q$-equivariant. So we get a $Q$-action on $\underline{\A}'_\hbar$. Let us produce a quantum
moment map for this action. For $\xi\in \q$ set $\underline{\Phi}(\xi)=\Phi^{\A}(\xi)-\Phi^{\W}(\xi)$.
The quantum moment map conditions for $\Phi^{\A},\Phi^{\W}$ imply that $\underline{\Phi}(\xi)$
commutes with $\W_\hbar(V)^{\wedge_0}$ and hence the image of $\underline{\Phi}$ is in $\underline{\A}'_\hbar$.
Now it is clear that $\underline{\Phi}:\q\rightarrow \underline{\A}'_\hbar$  is a quantum moment
map for the action of $Q$ on $\underline{\A}'_\hbar$.

We remark that the $Q$-action and the map $\underline{\Phi}$ are defined by the previous construction
uniquely up to an isomorphism of the form $\exp(\frac{1}{\hbar^2}\ad(z))$, where, in addition to conditions
mentioned above, $z$ is $Q$-invariant. Also we remark that if we have a $\K^\times$-action on $\A_\hbar$
as above but additionally commuting with $Q$ and such that $\Phi^{\A}(\xi)$ has degree $2$ for each $\xi$,
then the $Q$-action on $\underline{\A}'_\hbar$ also may be assumed to commute with $\K^\times$
and $\underline{\Phi}(\xi)$ may be assumed to have degree $2$.



\subsection{W-algebras via quantum slices}\label{SUBSECTION_W_slice}
We are going to consider a special case of the construction explained in the previous subsection.

Let $G$ be a simply connected semisimple algebraic group and let $\g$ be the Lie algebra of
$G$. Set $Y:=\g^*$. We remark that we can identify $\g$ with $\g^*$ by means
of the Killing form. Pick a nilpotent orbit $\Orb\subset\g$ and an element $e\in \Orb$.
We take $e$ for $\chi$. Let us equip $Y$ with a {\it Kazhdan} $\K^\times$-action defined as follows.
Pick an $\sl_2$-triple $(e,h,f)$, where $h$ is semisimple. Let $\gamma:\K^\times\rightarrow G$
be the one-parameter corresponding to $h$. We define a $\K^\times$-action on $\g^*$
by $t.\alpha=t^{-2}\gamma(t) \alpha, \alpha\in\g^*, t\in \K^\times$. We remark that $t.\chi=\chi$.

For $\A_\hbar$ we take the homogenized version $\U_\hbar$ of the universal enveloping algebra
defined by $\U_\hbar:=T(\g)[\hbar]/(\chi\otimes y-y\otimes \chi-\hbar^2 [\chi,y])$. We can  extend the Kazhdan
action to $\U_\hbar$. Explicitly, for $\xi\in \g$ with $[h,\xi]=i\xi$ we have $ t.\xi=t^{i+2}\xi$
and we set $t.\hbar:=t\hbar$.

Apply the construction of the previous subsection to $\A_\hbar,\chi$. We get the $\K[[\hbar]]$-algebra
$\underline{\A}'_\hbar$ acted on by $\K^\times$ together with a $\K^\times$-equivariant isomorphism
$\A^{\wedge_\chi}_\hbar=\W^{\wedge_0}_\hbar(V)\widehat{\otimes}_{\K[[\hbar]]}\underline{\A}'_\hbar$.
Here $V$ has the same meaning as before but we can describe it explicitly: namely, for $V$ we take
the subspace $[\g,f]\subset \g=T_\chi^*Y$.

Let $\underline{\A}_\hbar$ denote the subalgebra of all $\K^\times$-finite vectors in $\underline{\A}'_\hbar$.
It turns out that the algebra $\underline{\A}_\hbar/\hbar \underline{\A}_\hbar$, or, more precisely,
the corresponding variety is well-known in the theory of nilpotent orbits -- this is a so called
Slodowy slice. In more detail, set $S:=e+\z_\g(f)$ and view $S$ as a subvariety in $\g^*$ via
the identification $\g\cong \g^*$. Then $S$ is $\K^\times$-stable and is transverse to $Ge$ in $e$:
$\g=T_e S\oplus T_e Ge$. Moreover, the Kazhdan action contracts $S$ to $e$, meaning that $\lim_{t\rightarrow \infty}t.s=e$
for all $s\in S$. This implies that $S$ is transversal to any $G$-orbit it intersects -- $\g=T_s S+ T_s Gs$
for all $s\in S$. Also the contraction property means that the grading on $\K[S]$ induced
by the $\K^\times$-action is positive: there are no negative degrees, and the only elements
in degree 0 are constants.

The contraction property for the $\K^\times$-action on $S$ implies that the subalgebra of $\K[S]^{\wedge_\chi}=
\underline{\A}'_\hbar/\hbar \underline{\A}'_\hbar$ consisting of the $\K^\times$-finite elements
coincides with $\K[S]$. Moreover, since the degree of $\hbar$ is positive,
this implies that $\underline{\A}_\hbar/ \hbar\underline{\A}_\hbar=\K[S]$.

We set $\underline{\A}:=\underline{\A}_\hbar/(\hbar-1)\underline{\A}_\hbar$.
This is a filtered associative algebra whose associated graded is $\K[S]$.
The algebra $\underline{\A}_\hbar$ can be recovered as the Rees algebra
of $\underline{\A}$, while $\underline{\A}'_\hbar$ is the completion
$\underline{\A}_\hbar^{\wedge_\chi}$.

Let us remark that the algebra $\underline{\A}$ comes equipped with a homomorphism
$\Centr\rightarrow \underline{\A}$, where $\Centr$ is the center of the universal
enveloping algebra $\U(=\U_\hbar/(\hbar-1)\U_\hbar)$ of $\g$. Indeed,
set $\Centr_\hbar:=\U_\hbar^G$. This is the center of $\U_\hbar$. Consider $\Centr_\hbar$
as a subalgebra of $\A^{\wedge_\chi}_\hbar$. According to the previous
subsection,
we have $\Centr_\hbar\subset \underline{\A}'_\hbar$. It is easy to see that $\Centr_\hbar$ consists
of $\K^\times$-finite vectors so $\Centr_\hbar\subset \underline{\A}_\hbar$. This gives rise
to an embedding $\Centr\hookrightarrow \underline{\A}$. We remark that this embedding does not
depend on the choice (that of the embedding $V\hookrightarrow \tilde{I}^{\wedge_\chi}_\chi$)
we have made. This is because that the two choices are conjugate by an automorphism that commutes
with all elements of the center.

Consider the subgroup $Q:=Z_G(e,h,f)\subset G$. This group acts on $\U_\hbar$ and stabilizes $\chi$
(and $S$ as well). The $Q$-action commutes with $\K^\times$ and there is a quantum moment map
$\q\rightarrow \A_\hbar$ whose image consists of functions of degree $2$ with respect to the $\K^\times$-action.
So we get a $Q$-action on $\underline{\A}_\hbar$ as well as a quantum comoment map $\q\rightarrow \underline{\A}_\hbar$.
Both the action and the quantum moment map descend to $\underline{\A}$.

\subsection{Equivalence with a previous definition}\label{SUBSECTION_def_equiv}
Recall that the cotangent bundle  $T^*G$ carries a natural symplectic form $\omega$.
This form is invariant with respect to the natural $G\times G$-action. Moreover, for the
$\K^\times$-action by fiberwise dilations we have  $t.\omega= t^{-1}\omega, t\in \K^\times$.

We remark that we can trivialize $T^*G$ by using left-invariant 1-forms  hence $T^*G=G\times \g^*$.
Consider $S$ as a subvariety in $\g^*$ and set $X:=G\times S$. It turns out that the subvariety
$X\subset T^*G$ is symplectic. It is easy to see that $X$ is stable with respect to the left $G$-action,  as well as
to the Kazhdan  $\K^\times$-action on $T^*G$ given by $t.(g,\alpha)=(g\gamma(t)^{-1}, t^{-2}\gamma(t)\alpha)$.
Clearly, $\omega$ has degree $2$ with respect to the Kazhdan action. Also $Q$ acts on $T^*G$ by $q.(g,\alpha)=(gq^{-1}, q\alpha)$
and $X$ is $Q$-stable.

We remark that $\mu_G:T^*G\rightarrow \g^*, \mu_G(g,\alpha)=g\alpha$ is a moment map, i.e., a $G$-equivariant
map such that for any $\xi\in\g$ the derivation $\{\mu_G^*(\xi),\cdot\}$ of $\K[T^*G]$ coincides with the derivation produced
by $\xi$ via the $G$-action. Similarly, $\mu_Q: T^*G\rightarrow \q^*, \mu_Q(x,\alpha)=\alpha|_{\q}$.

In \cite{Wquant} the author proved that there is an associative product $*$ on $\K[X][\hbar]$, where $\hbar$
is an independent variable, satisfying the following properties:
\begin{enumerate}
\item $*$ is $G\times \K^\times$-equivariant, where $G\times\K^\times$ acts on $\K[X]$ as usual, $g.\hbar=\hbar, t.\hbar=t\hbar$.
\item For $f,g\in \K[X]$ we have $f*g=\sum_{i=0}^\infty D_i(f,g)\hbar^{2i}$, where $D_i$ is a bi-differential
operator of order at most $i$.
\item $f*g\equiv fg\mod \hbar^2$.
\item  $f*g-g*f\equiv \hbar^2\{f,g\}\mod\hbar^4$.
\item The map $\mu^*_G:\g\rightarrow \K[X][\hbar], \mu^*_Q: \q\rightarrow \K[X][\hbar]$ are quantum moment maps.
\end{enumerate}
The last property was established in \cite{HC}.  By definition, the W-algebra $\Walg$ is the quotient of the invariant subalgebra $\K[X][\hbar]^G$ by $\hbar-1$. The quantum moment map $\mu^*: \g\rightarrow \K[X][\hbar]$ gives rise
to the homomorphism $\Centr=U(\g)^G\rightarrow \Walg$.

\begin{Prop}\label{Prop:def_equiv}
We have a filtration preserving $Q$-equivariant isomorphism $\Walg\rightarrow \underline{\A}$
intertwining the embeddings of $\Centr$ and the quantum moment maps from $\q$.
\end{Prop}
\begin{proof}
Similarly to the above we have a star-product on $T^*G$ having the properties analogous
to (1)-(5). Set $x:=(1,\chi)\in X\subset T^*G$. Consider the completions $\K[T^*G][[\hbar]]^{\wedge_{Gx}},\K[X][[\hbar]]^{\wedge_{Gx}}$
of the corresponding algebras (w.r.t. star-products) at the ideals of $Gx$. According to Theorem 2.3.1 from \cite{HC},
we have a $G\times \K^\times$-equivariant (where we consider the Kazhdan $\K^\times$-actions)
topological $\K[[\hbar]]$-algebra isomorphism
$$\K[T^*G][[\hbar]]^{\wedge_{Gx}}\xrightarrow{\sim} \W^{\wedge_0}_\hbar(V)\widehat{\otimes}_{\K[[\hbar]]}\K[X][[\hbar]]^{\wedge_{Gx}},$$
and this isomorphism intertwines the quantum moment maps for the $G$-action and $Q$-action (the Weyl algebra component of
the quantum moment map for $G$ on the right hand side is 0). The algebra of $G$-invariants of the left hand side is $\U^{\wedge_\chi}_\hbar$,
while on the right hand side we get $\W^{\wedge_0}_\hbar(V)\widehat{\otimes}_{\K[[\hbar]]}\Walg^{\wedge_\chi}_\hbar$,
see {\it loc. cit.} So we can take $\Walg^{\wedge_\chi}_\hbar$ for $\underline{\A}'_\hbar$.
It follows that we have $Q$-equivariant isomorphisms $\Walg_\hbar\cong\underline{\A}_\hbar$ of graded $\K[\hbar]$-algebras
and $\Walg\cong \underline{\A}$ of filtered algebras. Both isomorphism intertwine the quantum comoment maps from $\q$.
Moreover, the embedding
$\U_\hbar^G\hookrightarrow \U_\hbar=\K[T^*G][[\hbar]]^G$ induced by the moment map is  just the inclusion.
This completes the proof of the proposition.
\end{proof}

\subsection{Map between the set of ideals}\label{SUBSECTION_ideals}
Let us construct the map $\bullet^\dagger$ mentioned in Theorem \ref{Thm:main}.
We will start with the general setting explained in Subsection \ref{SUBSECTION_quantum_slice}.

Consider the set $\Id_\hbar(\A_\hbar)$
of all $\K^\times$-stable $\hbar$-saturated ideals $\J_\hbar\subset \A_\hbar$, where ``$\hbar$-saturated''
means that $\A_\hbar/\J_\hbar$ is flat over $\K[\hbar]$. Similarly, consider the set $\Id_{\hbar}(\underline{\A}_\hbar')$ of all $D'$-stable $\hbar$-saturated ideals in $\underline{\A}_\hbar'$.
The discussion of  $D'$ in Subsection \ref{SUBSECTION_quantum_slice} implies that an $\hbar$-saturated
$\I'_\hbar\subset \underline{\A}'_\hbar$ is $D'$-stable if and only if $\W_\hbar^{\wedge_0}\widehat{\otimes}_{\K[[\hbar]]}\underline{\I}'_\hbar\subset \A_\hbar^{\wedge_\chi}$
is $D$-stable. In particular, the set $\Id_{\hbar}(\underline{\A}_\hbar')$ does not depend
on the choice of $D'$.

We have maps between $\Id_\hbar(\A_\hbar)$  and $\Id_\hbar(\underline{\A}'_\hbar)$
constructed as follows. Take an ideal $\J_\hbar\subset \A_\hbar$ and form its closure $\J_\hbar^{\wedge_\chi}
\subset \A_\hbar^{\wedge_\chi}$. This ideal is $D$-stable but also one can check that it is actually
$\hbar$-saturated. As such, the ideal $\J_\hbar^{\wedge_\chi}$ has the form $\W_\hbar^{\wedge_0}\widehat{\otimes}_{\K[[\hbar]]}\I'_\hbar$ for a unique two-sided ideal
$\I'_\hbar$ in $\underline{\A}'_\hbar$. The ideal $\I'_\hbar$ is automatically $D'$-stable
and $\hbar$-saturated. We consider the map $\bullet_\dagger:\Id_\hbar(\A_\hbar)\rightarrow
\Id_\hbar(\underline{\A}_\hbar')$ sending $\J_\hbar$ to $\I'_\hbar$.

Let us produce a map
in the opposite direction. Take $\I'_\hbar\in \Id_\hbar(\underline{\A}'_\hbar)$. Then $\J_\hbar:=\A_\hbar\cap
\W^{\wedge_0}_\hbar\widehat{\otimes}_{\K[[\hbar]]}\I'_\hbar$ is a $\K^\times$-stable $\hbar$-saturated
ideal in $\A_\hbar$. Consider the map $\bullet^\dagger:\Id_\hbar(\underline{\A}'_\hbar)\rightarrow \Id_{\hbar}(\A_\hbar)$ sending $\I_\hbar'$ to $\J_\hbar$.

Now suppose that the grading on $\K[Y]$ induced by the $\K^\times$-action is positive. Then $\Id_\hbar(\A_\hbar)$
is in bijection with the set $\Id(\A)$ of two-sided ideals in $\A:=\A_\hbar/(\hbar-1)$. Under this bijection,
the ideal in $\A$ corresponding to $\J_\hbar\in \Id_\hbar(\A_\hbar)$ is $\J_\hbar/(\hbar-1)\J_\hbar$.

Similarly, suppose that $D'$ is also induced from some $\K^\times$-action such that $\underline{\A}_\hbar'$ is the projective limit of some positively graded algebras (this is the case in the situation considered in Subsection \ref{SUBSECTION_W_slice}).
Then we have natural identifications $\Id_\hbar(\underline{\A}_\hbar')\cong \Id_\hbar(\underline{\A}_\hbar)\cong \Id(\underline{\A})$. So we have maps between $\Id(\A),\Id(\underline{\A})$ that still will be denoted
by $\bullet_\dagger,\bullet^\dagger$.

In a special case we have some additional information
about the maps $\bullet_\dagger,\bullet^\dagger$. Suppose that $Y$ is still equipped with
a contracting $\K^\times$-action and, moreover, has only finitely many symplectic leaves.
For a symplectic leaf $\Leaf$  let $\Id_{\Leaf}(\A_\hbar)$ denote the
subset of $\Id(\A)$ consisting of all ideals $\J$ such that $\gr(\A/\J)$
is supported on the closure of $\Leaf$. The maximal elements in $\Id_\Leaf(\A)$ are precisely prime (=primitive
by \cite{ES_appendix}) ideals.

Now let $\Leaf$ be the leaf containing $\chi$.  Then $\bullet_\dagger$
defines a map  $\Id_\Leaf(\A)\rightarrow \Id_{\hbar,fin}(\underline{\A}_\hbar')$, where, by definition,
the target set  consists of
all ideals $\I'_\hbar$ such that $\underline{\A}'_\hbar/\I'_\hbar$ is free of finite
rank over $\K[[\hbar]]$.  For a prime ideal $\J\in \Id_\Leaf(\A)$ and any minimal prime
ideal $\I'_\hbar$ of $\J_\dagger$ we have $\J=(\I'_\hbar)^\dagger$, see \cite{ES_appendix}, Lemma A4.
Corollary 3.17 from \cite{ES} implies that there are finitely many prime
ideals in $\Id_{\hbar,fin}(\underline{\A}'_\hbar)$ and so $\Id_\Leaf(\A)$ also contains
 finitely many prime ideals.

We are interested in the special case when $\A$ is a central reduction of $\U$
at some central character $\lambda$. Consider a unique $\K^\times$-equivariant $\K[\hbar]$-linear
homomorphism $Z(\U_\hbar)\rightarrow \K[\hbar]$ specializing to $\lambda$
at $\hbar=1$. This homomorphism will also be denoted by $\lambda$.
So $\A_\hbar=\U_{\lambda,\hbar}$ is the Rees algebra of $\U_\lambda$.
The underlying variety $Y$ is the nilpotent cone $\mathcal{N}$ of $\g$ and so contains
finitely many symplectic leaves (=nilpotent orbits). Using the usual
$\K^\times$-action we get an identification $\Id(\U_\lambda)\cong \Id_\hbar(\U_{\lambda,\hbar})$.
Similarly, using the Kazhdan action we get an identification $\Id(\Walg_\lambda)\cong \Id_\hbar(\Walg_{\lambda,\hbar})$.
We remark that the Kazhdan action differs from the usual one by inner automorphisms
and so an $\hbar$-saturated ideal in $\U_{\lambda,\hbar}$ is stable under the usual $\K^\times$-action
if and only if it is stable under the Kazhdan $\K^\times$-action. Since the
Kazhdan action is contracting, we get $\Id_\hbar(\Walg_{\lambda,\hbar})\cong \Id_\hbar(\Walg_{\lambda,\hbar}^{\wedge_\chi})$.

So we get maps $\bullet_\dagger: \Id(\U_\lambda)\rightarrow \Id(\Walg_\lambda),
\bullet^\dagger: \Id(\Walg_\lambda)\rightarrow \Id(\U_\lambda)$ that first appeared
in \cite{Wquant}. Their properties are summarized in the following proposition.

\begin{Prop}\label{Prop:map_summary}
Let $\g\cong \sl_n$.
\begin{enumerate}
\item The sets $\Prim(\Walg_\lambda),\Prim(\U_\lambda)$ are finite.
\item The map $\bullet^\dagger:\Prim(\Walg_\lambda)\rightarrow \Prim(\U_\lambda)$
is surjective. More precisely, for $\J\in \Prim_{\Orb'}(\U_\lambda)$ there is
$\I\in \Prim_{\Leaf'}(\Walg_\lambda)$ with $\I^\dagger=\J$. Here $\Leaf'$
is an irreducible component of $S\cap \Orb'$ (in fact, below we will see
that $S\cap\Orb'$ is irreducible).
\item The restriction of $\bullet^\dagger$ to $\Prim_{\chi}(\Walg_\lambda)$ is a bijection
$\Prim_\chi(\Walg_\lambda)\xrightarrow{\sim}\Prim_{\Orb}(\U_\lambda)$.
\end{enumerate}
\end{Prop}
\begin{proof}
The first claim for $\U_\lambda$ is well known. For $\Walg_\lambda$ it follows from
the main result of \cite{ES_appendix}. Indeed, the associated graded algebra $\gr\Walg_\lambda$
is nothing else but $\K[S\cap\mathcal{N}]$, where $\mathcal{N}$ is the nilpotent cone
in $\g$. The Poisson variety $S\cap\mathcal{N}$ has finitely many symplectic leaves
(according to Proposition \ref{Prop:leaves} below -- that is independent of the present proposition-- these leaves are in bijection with
nilpotent orbits $\Orb_1\subset \g$ with $\Orb\subset \overline{\Orb}_1$). So the main theorem
of \cite{ES_appendix} does apply to $\Walg_\lambda$.

Assertion (2) follows from \cite{Wquant}, Theorem 1.2.2(vii). In fact, the associated
variety of any primitive ideal in $\Walg_\lambda$ is irreducible thanks to a theorem
of Ginzburg, \cite{Ginzburg_irr}.

Finally (3) follows from
\cite{HC}, Conjecture 1.2.1 (proved in that paper) because the action of the component
group on $\Prim_{\chi}(\Walg_\lambda)$ is trivial.
\end{proof}

We remark that $\Prim(\U)=\bigsqcup_\lambda \Prim(\U_\lambda)$, where the union is taken over all
central characters. A similar claim holds for $\Walg$. The map $\bullet^\dagger:\Prim(\Walg)\rightarrow \Prim(\U)$
is obtained from  the maps $\bullet^\dagger: \Prim(\Walg_\lambda)\rightarrow \Prim(\U_\lambda)$.

\section{Proof of the main theorem}
\subsection{Classical level}
Here we are going to prove a ``quasi-classical'' analog of Theorem \ref{Thm:main} that seems
to be of some independent interest. Instead of primitive ideals in associative algebras
we will consider symplectic leaves of the corresponding Poisson varieties.

From the description of the Poisson structure on $S$ given in \cite{GG},  symplectic
leaves are irreducible components of $S\cap \Orb_1$, where $\Orb_1$ are (co)adjoint
orbits in $\g\cong \g^*$.

The main result of this section characterizes nilpotent symplectic leaves of $S$.

\begin{Prop}\label{Prop:leaves}
Suppose $\g=\sl_n$.
Let $\Orb_1$ be a nilpotent orbit with $\Orb\subset \overline{\Orb}_1$. Then the intersection
$\Orb_1\cap S$ is irreducible.
\end{Prop}
\begin{proof}
We will check that the intersection $S\cap \overline{\Orb}_1$ is normal. Since $S\cap \overline{\Orb}_1$
is $\K^\times$-stable and hence connected, the normality implies that $S\cap\overline{\Orb}_1$
is irreducible. Being an open subvariety in $S\cap \overline{\Orb}_1$, the variety $S\cap \Orb_1$
is also irreducible.

Again, thanks to the contracting $\K^\times$-action it is enough to show that the completion
$(S\cap \overline{\Orb}_1)^{\wedge_e}$ at the point $e$ is normal. Recall that the intersection
$S\cap \Orb$ is transversal at $e$. This implies that $\overline{\Orb}_1^{\wedge_e}$ decomposes
into the direct product $\Orb^{\wedge_e}\times (S\cap \overline{\Orb}_1)^{\wedge_e}$. According to
Kraft and Procesi, \cite{KP}, the variety $\overline{\Orb}_1$ is normal. Hence the formal
scheme $\overline{\Orb}_1^{\wedge_e}$ is normal as well. The direct product decomposition
now implies that  $(S\cap \overline{\Orb}_1)^{\wedge_e}$ is normal.
\end{proof}

\begin{Rem}\label{Rem:all_leaves}
In fact, the techniques used below to prove Theorem \ref{Thm:main} allow one
to show that $S\cap \Orb_1$ is irreducible for any (not necessarily nilpotent) orbit $\Orb_1$.
\end{Rem}

\subsection{Quantum level}
Pick some nilpotent orbit $\widetilde{\Orb}\subset \g$ whose closure contains $\Orb$,
let $\tilde{\chi}\in \widetilde{\Orb}\cap S$, and
let $\widetilde{\Walg}$ be the corresponding W-algebra. We claim that the inequalities
\begin{equation}\label{eq:inequal}|\Prim_{\widetilde{\Orb}\cap S}(\Walg_\lambda)|
\leqslant |\Prim_{\tilde{\chi}}(\widetilde{\Walg}_{\lambda})|\end{equation}
(for all possible $\widetilde{\Orb}$) imply that the map $\bullet^\dagger$ sends
$\Prim_{\widetilde{\Orb}\cap S}(\Walg_\lambda)$ to $\Prim_{\widetilde{\Orb}}(\U_\lambda)$ and is a bijection between the two
sets.

Indeed, thanks to Proposition \ref{Prop:map_summary}(3), the map $\bullet^{\dagger}:\Prim_{\tilde{\chi}}(\widetilde{\Walg}_{\lambda})\rightarrow
\Prim_{\widetilde{\Orb}}(\U_\lambda)$ is a bijection. On the other hand,
the preimage of $\J\in \Prim_{\widetilde{\Orb}}(\U_\lambda)$ in $\Prim(\Walg_\lambda)$
contains at least one element from $\Prim_{\widetilde{\Orb}\cap S}(\Walg_\lambda)$, assertion (2).
Finally, by assertion (1), both sets $\Prim(\Walg_\lambda),\Prim(\U_\lambda)$ are finite.
So  inequalities (\ref{eq:inequal}) imply the claim of the previous paragraph.

To prove the inequality $|\Prim_{\widetilde{\Orb}\cap S}(\Walg_\lambda)|
\leqslant |\Prim_{\tilde{\chi}}(\widetilde{\Walg}_{\lambda})|$ we will apply the general construction
of Subsection \ref{SUBSECTION_quantum_slice} to $Y=S\cap\mathcal{N},\A=\Walg_\lambda$ and the point $\tilde{\chi}\in S\cap \widetilde{\Orb}$.

\begin{Lem} $\underline{\A}'_\hbar\cong \widetilde{\Walg}_{\lambda\hbar}^{\wedge_{\tilde{\chi}}}$, where $\widetilde{\Walg}_{\lambda\hbar}$
is the homogenized version (=the Rees algebra) of the central reduction $\widetilde{\Walg}_{\lambda}$ of $\widetilde{\Walg}$.
\end{Lem}
\begin{proof}
Set $V_1:=T_{\tilde{\chi}}(S\cap \widetilde{\Orb})$ and let $V_2$ be the skew-orthogonal complement of $V_1$
in $T_{\tilde{\chi}}\widetilde{\Orb}$. We can form the corresponding completed homogenized Weyl algebras
$\W_\hbar(V_1)^{\wedge_0},\W_\hbar(V_2)^{\wedge_0}$. Then we have
\begin{align}\label{eq:decomp_11}
&\U_{\lambda\hbar}^{\wedge_{\tilde{\chi}}}=\W_{\hbar}(V_1\oplus V_2)^{\wedge_0}\widehat{\otimes}_{\K[[\hbar]]}\widetilde{\Walg}_{\lambda\hbar}^{\wedge_{\tilde{\chi}}},\\\label{eq:decomp12}
&\Walg_{\lambda\hbar}^{\wedge_{\tilde{\chi}}}=\W_\hbar(V_1)^{\wedge_0}\widehat{\otimes}_{\K[[\hbar]]}\underline{\A}'_\hbar.
\end{align}

Applying the construction used in the proof of Proposition \ref{Prop:def_equiv}
to the point $(1,\tilde{\chi})\in G\times S\subset T^*G$ we see that $\U_\hbar^{\wedge_{\tilde{\chi}}}\cong\W_\hbar(V_2)^{\wedge_0}\widehat{\otimes}_{\K[[\hbar]]}\Walg_\hbar^{\wedge_{\tilde{\chi}}}$.
From here and the description of the embedding $\Centr\rightarrow \Walg$ provided in Subsection \ref{SUBSECTION_def_equiv}
one can see that
\begin{equation}\label{eq:decomp_13}
\U_{\lambda\hbar}^{\wedge_{\tilde{\chi}}}\cong \W_\hbar(V_2)^{\wedge_0}\widehat{\otimes}_{\K[[\hbar]]} \widetilde{\Walg}_{\lambda\hbar}^{\wedge_{\tilde{\chi}}}.
\end{equation}
Combining (\ref{eq:decomp12}) and (\ref{eq:decomp_13}), we get
\begin{equation}\label{eq:decomp14}
\U_{\lambda\hbar}^{\wedge_{\tilde{\chi}}}\cong \W_\hbar(V_1\oplus V_2)^{\wedge_0}\widehat{\otimes}_{\K[[\hbar]]}\underline{\A}'_\hbar.
\end{equation}
From Subsection \ref{SUBSECTION_quantum_slice}, we see that
there is a $\K[[\hbar]]$-linear isomorphism $\widetilde{\Walg}_{\lambda\hbar}^{\wedge_{\tilde{\chi}}}\cong \underline{\A}'_\hbar$.
\end{proof}

Now we have two derivations of $\widetilde{\Walg}_{\lambda\hbar}$, the derivation $\tilde{D}$ induced by the Kazhdan action
defined for the nilpotent element $\tilde{\chi}$, and the derivation $D'$ coming from an isomorphism
$\widetilde{\Walg}_{\lambda\hbar}^{\wedge_{\tilde{\chi}}}\cong \underline{\A}'_\hbar$. Both satisfy $\tilde{D}\hbar=D'\hbar=\hbar$.
Consider the sets $\widetilde{\Prim}_{fin,\hbar}(\widetilde{\Walg}_{\lambda\hbar}^{\wedge_{\tilde{\chi}}}),
\Prim'_{fin,\hbar}(\widetilde{\Walg}_{\lambda\hbar}^{\wedge_{\tilde{\chi}}})$
that consist of all prime (=maximal) $\hbar$-saturated ideals $\I'_\hbar\subset \widetilde{\Walg}_{\lambda\hbar}^{\wedge_{\tilde{\chi}}}$
such that $\widetilde{\Walg}_{\lambda\hbar}^{\wedge_{\tilde{\chi}}}/\I'_\hbar$ is of finite rank over $\K[[\hbar]]$
and such that $\I'_\hbar$ is, respectively, $\tilde{D}$- and $D'$-stable. The set $\widetilde{\Prim}_{fin,\hbar}(\Walg_{1\lambda\hbar}^{\wedge_{\tilde{\chi}}})$ is in natural bijection with $\Prim_{fin}(\widetilde{\Walg}_{\lambda})$. On the other hand, by the results recalled in Subsection \ref{SUBSECTION_ideals},
the cardinality of $\Prim'_{fin,\hbar}(\widetilde{\Walg}_{\lambda\hbar}^{\wedge_{\tilde{\chi}}})$ is bigger than or equal
to that of $\Prim_{S\cap \widetilde{\Orb}}(\Walg)$. So it remains to show that the two sets coincide.

It is enough to check that any derivation $d$ of $\underline{\A}'_\hbar=\widetilde{\Walg}_{\lambda\hbar}^{\wedge_{\tilde{\chi}}}$
with $d(\hbar)=\hbar$ fixes any maximal $\hbar$-saturated ideal of finite corank.  Consider the quotient
$(\underline{\A}_\hbar')^{(n)}$ of $\underline{\A}'_\hbar$ by the ideal generated by the elements
$s_{2n}(x_1,\ldots,x_{2n})=\sum_{\sigma\in \mathfrak{S}_{2n}}\operatorname{sgn}(\sigma)x_{\sigma(1)}\ldots x_{\sigma(2n)},  x_1,\ldots,x_{2n}\in \underline{\A}_\hbar'$. This ideal is clearly $d$-stable.
Also consider the analogous quotient $\underline{\A}_\hbar^{(n)}$ of $\underline{\A}_\hbar:=\widetilde{\Walg}_{\lambda\hbar}$.
It follows from Section 7.2 of \cite{Miura} that $\underline{\A}_\hbar^{(n)}$ has finite rank over $\K[[\hbar]]$.
But $(\underline{\A}'_\hbar)^{(n)}$ is the completion of $\underline{\A}_\hbar^{(n)}$ at $\tilde{\chi}$.
So $(\underline{\A}'_\hbar)^{(n)}$ has finite rank over $\K[[\hbar]]$. Therefore the localization
$(\underline{\A}'_\hbar)^{(n)}[\hbar^{-1}]$ is a finite dimensional $\K[\hbar^{-1},\hbar]]$-algebra.

Maximal $\hbar$-saturated ideals of finite corank in $\A'_\hbar$ are in a natural one-to-one correspondence
with maximal ideals of finite codimension in the $\K[\hbar^{-1},\hbar]]$-algebra $\A'_\hbar[\hbar^{-1}]$.
Clearly $\A'_\hbar[\hbar^{-1}]^{(n)}=(\underline{\A}'_\hbar)^{(n)}[\hbar^{-1}]$. Thanks to the Amitsur-Levitzki
theorem, every maximal ideal of finite codimension in $\A'_\hbar[\hbar^{-1}]$ is the preimage of an ideal
in  $\A'_\hbar[\hbar^{-1}]^{(n)}$ for some $n$. Of course, $d$ induces a $\K[\hbar^{-1},\hbar]]$-linear derivation
of $\A'_\hbar[\hbar^{-1}]^{(n)}$. Now it remains to use a  fact that a maximal ideal in a finite dimensional
algebra is stable under any derivation of this algebra. For reader's convenience we will provide a proof
here.

Let $A$ be a finite dimensional algebra over some field $K$ and let $\m$ be its maximal ideal.
Replacing $A$ with $A/\bigcap_{i=1}^\infty \m^i$, we may assume that $\m$ is a nilpotent ideal
and hence the radical of $A$. To complete the proof apply Lemma 3.3.3 from \cite{Dixmier}.

\end{document}